\documentclass{article}
\usepackage{amsmath,amssymb,latexsym,graphics,epsfig}
\usepackage{color}

\usepackage{amsthm}
\usepackage{graphicx,url}
\usepackage{hyperref}
\usepackage[algo2e,vlined,ruled]{algorithm2e}

\def\1{\mathbbm{1}}

\theoremstyle{definition} 

\theoremstyle{plain} 
\newtheorem{theorem}{Theorem}
\newtheorem{lemma}{Lemma}
\newtheorem{corollary}{Corollary}
\newtheorem{claim}{Claim}
\newtheorem{observation}{Observation}
\newtheorem{proposition}{Proposition}
\newtheorem{que}{Question}
\newtheorem{remark}{Remark}

\def \0{{\bf 0}}

\title
{\bf On the Wiener index, distance cospectrality and transmission-regular graphs}

\author{Aida Abiad\thanks{Department of Quantitative Economics, Maastricht University, Maastricht, The Netherlands (A.AbiadMonge@maastrichtuniversity.nl).}\and Boris Brimkov\thanks{Department of Computational and Applied Mathematics, Rice University, USA (boris.brimkov@rice.edu).}\and Aysel Erey\thanks{Department of Mathematics,
University of Denver, USA (aysel.erey@du.edu).}\and Lorinda Leshock\thanks{Department of Mathematical Sciences, University of Delaware, USA (lleshock@udel.edu).}\and
Xavier Mart\'{i}nez-Rivera\thanks{Department of Mathematics, Iowa State University, USA (xaviermr@iastate.edu).}\and  Suil O\thanks{ SRC-Applied Algebra and Optimization Research Center, Sungkyunkwan University, Korea (suilo@skku.edu).}\and  Sung-Yell Song\thanks{Department of Mathematics, Iowa State University, USA (sysong@iastate.edu).}\and Jason Williford\thanks{Department of Mathematics, University of Wyoming, USA (jwillif1@uwyo.edu).}}

\date{}

\begin{document}
\maketitle

\begin{abstract}
\noindent
In this paper, we investigate various algebraic and graph theoretic properties of the distance matrix of a graph. Two graphs are $D$-cospectral if their distance matrices have the same spectrum. We construct infinite pairs of $D$-cospectral graphs with different diameter and different Wiener index. A graph is $k$-transmission-regular if its distance matrix has constant row sum equal to $k$. We establish tight upper and lower bounds for the row sum of a $k$-transmission-regular graph in terms of the number of vertices of the graph. Finally, we determine the Wiener index and its complexity for linear $k$-trees, and obtain a closed form for the Wiener index of block-clique graphs in terms of the Laplacian eigenvalues of the graph. The latter leads to a generalization of a result for trees which was proved independently by Mohar and Merris.
\\[7pt]

\noindent {\textbf{Keywords:} distance matrix, distance cospectral graphs, diameter, Wiener index, Laplacian matrix, transmission-regular.\\}
\end{abstract}

\section{Introduction and Preliminaries}
Every graph $G$ in this paper is undirected, simple and loopless and has $n$ vertices.
The \emph{distance matrix} $D$ of a connected graph $G$ is the (symmetric) matrix indexed by the vertices of $G$ and with its $(i,j)$-entry $d_{ij}$ equal to the distance  between the vertices $v_i$ and $v_j$, i.e., the length of a shortest path between $v_i$ and $v_j$.
After its application in 1971 by Graham and Pollack \cite{GP1971} as a tool to study a data communication problem, the distance matrix of a connected graph eventually became a topic of interest when researchers tried to compute its characteristic polynomial.
That naturally led to the study of its eigenvalues
(or its spectrum); for a survey on distance spectra and recent results on this topic, see \cite{AH14, D2012, KST2016, SI2011}.

Studying the eigenvalues of a matrix associated with a graph is the subject of spectral graph theory, where the main objective is determining what characteristics of the graph are reflected in the spectrum of the matrix under consideration. One way to do this is to study the relationships between \emph{cospectral} graphs, that is, graphs whose associated matrices share a common spectrum; such pairs of graphs help us understand how limited the information that can be extracted from the spectrum is. Although the matrix of interest here is the distance matrix, for the sake of clarity, we will call two cospectral graphs \emph{$D$-cospectral} when their distance matrices have the same spectrum.
Moreover, the set of eigenvalues of the distance matrix of a connected graph $G$, denoted by $\lambda_1\geq \cdots \geq \lambda_n$, is called the \emph{distance spectrum} of $G$ (abbreviated \emph{$D$-spectrum}).
In 1977, McKay \cite{MK1977} showed that hardly any trees can be identified by their $D$-spectrum. Recently, it was established that $D$-cospectral graphs may have a differing number of edges, and a method to construct such graphs was presented in \cite{K2016}. It turns out that for the distance matrix, things can get rather complicated, and there are many open problems regarding what properties follow from the distance spectrum. One open question has been to determine whether or not two $D$-cospectral graphs always have the same diameter (that is, the largest distance between any two vertices of the graph). This question is answered in Section \ref{sec:Dcospectral}.

The \emph{Wiener index} $W$ is a topological index used in theoretical chemistry as a structural descriptor for organic molecules. For a connected graph $G$ with distance matrix $D=(d_{ij})$, the Wiener index $W$ is defined as follows:
$$W(G)=\frac{1}{2}\displaystyle \sum_{i=1}^{n}\sum_{j=1}^{n}d_{ij},$$
or equivalently, $W$ is the sum of distances between all pairs of vertices of $G$. For a survey on the Wiener index, see \cite{Survey-Wiener2001}.
This 70-year-old parameter remains relevant today, which is evidenced by the strong attention it continues to receive
(see the recent papers
\cite{W-Ref-1, W-Ref-2, W-Ref-3, W-Ref-4, W-Ref-5, W-Ref-6, W-Ref-7, W-Ref-8, W-Ref-9}).
As stated in \cite{Survey-Wiener2001}, the Wiener index of a graph also has applications in communication, facility location, and cryptography, among others.  From a spectral point of view, the Wiener index of a graph has also been connected to the distance spectral radius of the graph in \cite{I2009}.

Let $i,j,k$ be nonnegative integers; a graph $G$ is \emph{distance-regular} if for any choice of $u,v \in V(G)$ with $d(u,v) = k$, the number of vertices $w \in V(G)$ such that $d(u,w)=i$ and $d(v,w)=j$ is independent of the choice of $u$ and $v$. It follows from \cite[Theorem 4]{W and dist reg} that if two distance-regular graphs are $D$-cospectral, then they must have the same Wiener index. Hence, a natural question arises: do $D$-cospectral graphs have the same Wiener index? This question is answered in Section \ref{sec:Dcospectral}.

We say that a graph is \emph{$k$-transmission-regular} (or \emph{transmission-regular}) if its distance matrix has constant row sum equal to $k$. Naturally, just as regular graphs, transmission-regular graphs are also of interest in spectral graph theory.
Transmission-regular graphs were introduced by Handa \cite{Handa1999} in 1999, and in 2009, Balakrishnan et al. \cite{Balakrishnan2009} showed that this class of graphs is the same as the class of (connected) \emph{distance-balanced} graphs \cite{dist-balanced}.
These graphs have applications in chemistry, as they were used in \cite{dist-balanced} to construct an infinite family of graphs that maximize the Szeged index, a graph invariant in chemical graph theory.
It is well-known that if a connected graph on $n \geq 3$ vertices is $k$-regular, then $2 \leq k \leq  n-1$, where both bounds are sharp.
In Section \ref{sec:transmissionregular}, we establish an analogous result for $k$-transmission-regular graphs, where it is shown that
$n-1 \leq k \leq \lfloor{\frac{n^2}{4}}\rfloor$, with both bounds being sharp.

The \emph{Laplacian matrix} of a graph $G$ is defined as $L=\Delta -A$, where $A$ is the adjacency matrix of $G$ and $\Delta$ is the degree matrix. The Laplacian eigenvalues of $G$ are denoted by $\mu_1\geq \mu_2\geq \cdots \geq \mu_n=0$. Mohar \cite{M1991} and Merris \cite{Me1989,Me1990} independently proved that there exists a relation between the Wiener index of a tree and its Laplacian eigenvalues. As stated in \cite{Survey-Wiener2001}, this result ignited hope in seeing linear algebra become a tool in the arsenal for studying the theory of the Wiener index.
As a contribution in this direction, in Section \ref{sec:Wtreelikegraphs} we generalize the result of Mohar and Merris for
\emph{block-clique graphs}, that is, the connected graphs in which every block is a clique (our result considers the case where the cliques are of the same size).

The paper concludes by addressing the computational aspect of the Wiener index at the end of Section \ref{sec:Wtreelikegraphs}, where an algorithm that calculates this index for a family of tree-like graphs is presented.

The topics considered in this paper -- the Wiener index, the $D$-spectrum of a graph, and transmission-regular graphs -- are all closely related. Transmission-regular graphs are used in \cite{AP15}
to construct graphs with a particular number of distinct $D$-eigenvalues.
In \cite{AP15},  transmission-regular graphs are also used to
construct graphs with few distinct $D$-eigenvalues and arbitrary diameter.
Given a graph $G$ on $n$ vertices with $D$-eigenvalues
$\lambda_1, \lambda_2, \dots, \lambda_n$,
the {\em distance Estrada index} of $G$ is $\sum_{j=1}^{n}e^{\lambda_j}$.
Upper and lower bounds for the distance Estrada index of a graph are
given in \cite{S16} in terms of its Wiener index and its diameter.
A strict lower bound for the distance spectral radius of a tree in terms of
the Wiener index was established in \cite[Theorem 4]{Z07};
the same lower bound was later shown to
hold for any graph in \cite[Corollary 7]{ZT07},
with equality holding if and only if the graph is transmission-regular.

\section{$D$-cospectral graphs, diameter and Wiener index}\label{sec:Dcospectral}
In this section we investigate algebraic and graph theoretic properties of the $D$-spectrum
of a graph and the Wiener index and diameter of the graph. In particular, we settle the following two open questions:

\begin{que}\label{diameter}
\emph{Do $D$-cospectral graphs have the same diameter?}
\end{que}

\begin{que}\label{Wiener}
\emph{Do $D$-cospectral graphs have the same Wiener index?}
\end{que}

\noindent It is easy to see that the converse to Question \ref{Wiener} is false, Figure \ref{fig:sameWienerNotCospectral} shows the smallest counterexample on four vertices.
\begin{figure}[h!]
\centering
  \includegraphics[scale=0.3]{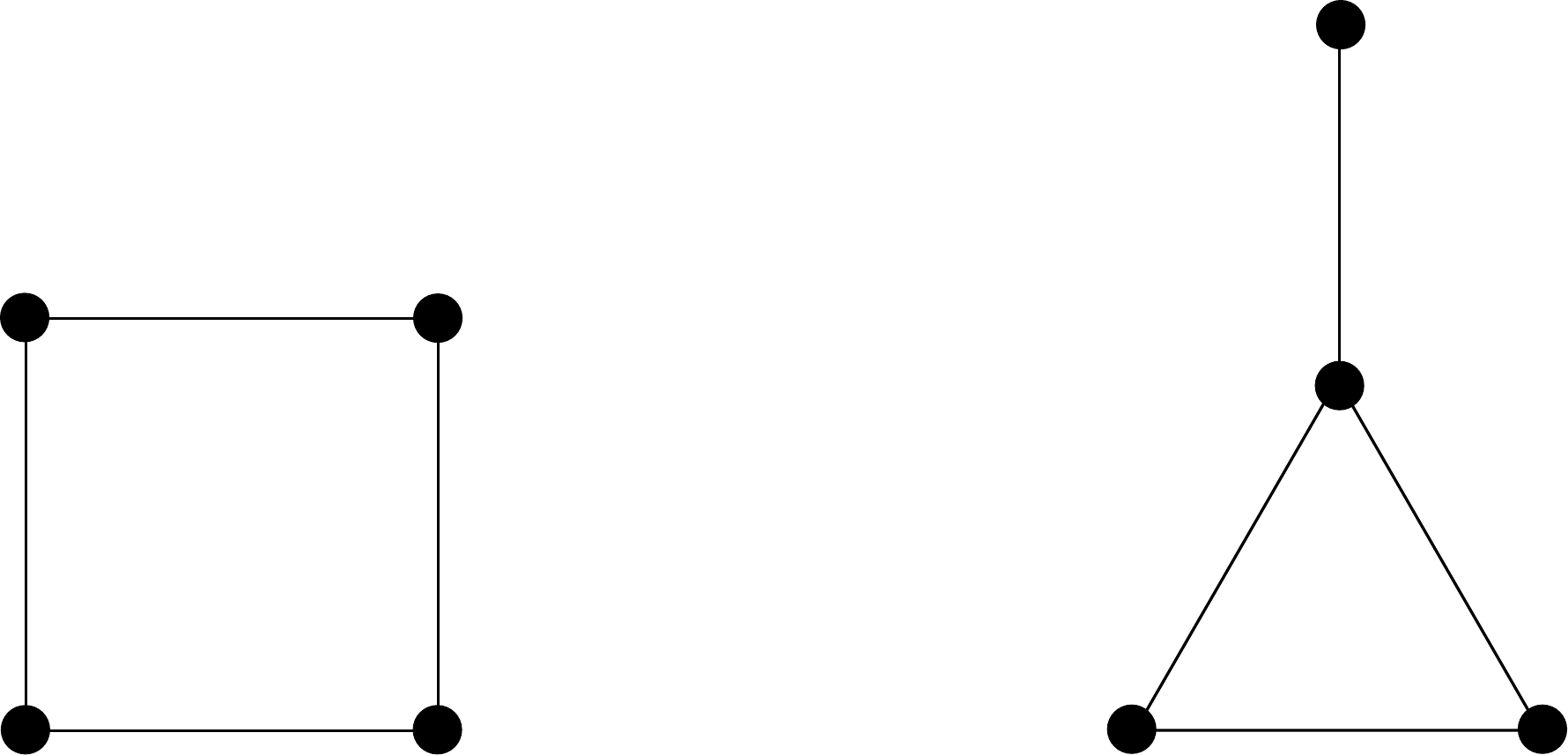}
  \caption{Two graphs on four vertices which are not $D$-cospectral but have the same Wiener index $W=8$.}
  \label{fig:sameWienerNotCospectral}
\end{figure}

Moreover, SAGE simulations on graphs with up to ten vertices confirm that the pair of graphs shown in Figure \ref{fig:DCospectralDIFERENTdiameterANDWiener} is the smallest and unique pair of $D$-cospectral graphs having different diameter and Wiener index.

\begin{center}
\begin{figure}[h!]
\centering
  \includegraphics[scale=0.25]{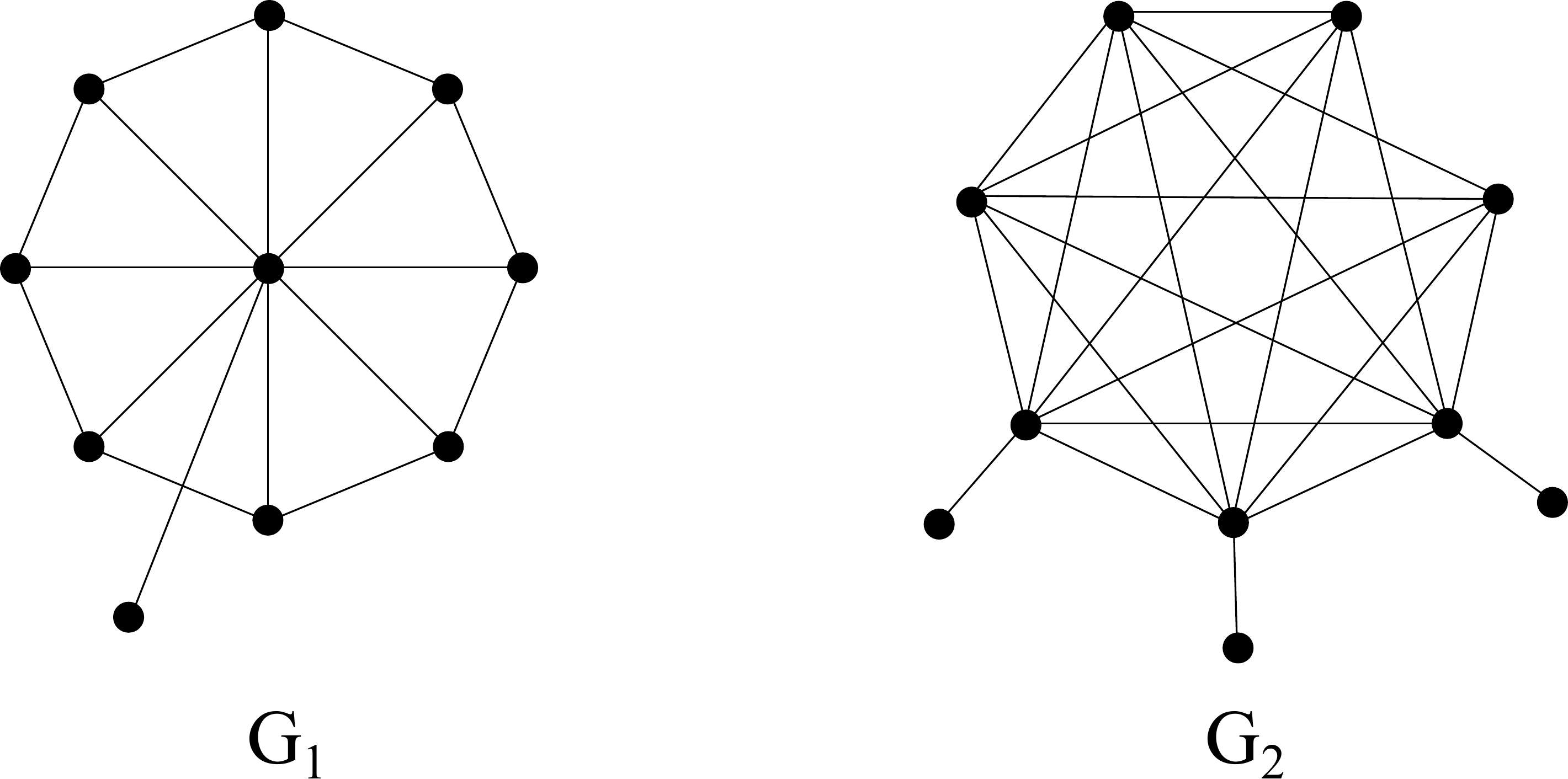}
    \caption{Unique pair of $D$-cospectral graphs with ten vertices or less having different diameter and different Wiener index.}
  \label{fig:DCospectralDIFERENTdiameterANDWiener}
\end{figure}
\end{center}


Let $G$ be a graph with vertex set $V$, adjacency matrix $A$ and distance matrix $D$, and let $q$ be a positive integer.  We define the $q$-\textit{coclique extension} of $G$ to be the graph $G_q$ with vertex set $V \times \{1,\dots,q\}$, where $(x,i)$ is adjacent to $(y,j)$ if and only if $x$ is adjacent to $y$ in $G$.  Similarly we define the $q$-\textit{clique extension} of $G$ to be the graph $G_q^+$ with vertex set $V \times \{1,\dots,q\}$ with $(x,i)$ adjacent to $(y,j)$ if and only if $x$ is adjacent to $y$ in $G$ or $x = y$ and $i \neq j$. See Figure \ref{fig:qcliqueextension} for an illustration. The adjacency matrices of the graphs $G_q$ and $G_q^+$ are easily seen to be
$\begin{pmatrix} J_q \otimes A  \end{pmatrix}$ and $\begin{pmatrix}  (J_q \otimes (A+I)) - I \end{pmatrix}$, respectively, where $J_q$ is the $q$ by $q$ matrix of all 1's, $I$ is an identity matrix of the appropriate size, and $\otimes$ denotes the \emph{Kronecker} product. The $q$-(co)clique extension of $G$ is also known as blow up of $G$ to $q$-(co)cliques. The following lemma, which is also a corollary of (the proof of) Theorem 3.3. in \cite{I2009coclique}, gives an expression for the distance matrices of $G_q$ and $G_q^+$. Note that $q$-(co)clique extensions of $G$ may be seen also as the lexicographic products of $G$ with the empty graph $N_q$ and the complete graph $K_q$, respectively. However, we keep the following result and its proof for the sake of simplicity.

\begin{lemma}\label{dmat}

Let $G$ be a connected graph with at least two vertices.  Then the distance matrix of $G_q$ is $\begin{pmatrix} J_q \otimes D + (J-I)_q \otimes 2I   \end{pmatrix}$, and the distance matrix of
$G_q^+$ is $\begin{pmatrix} J_q \otimes D + (J-I)_q \otimes I   \end{pmatrix}$.

\end{lemma}

\begin{proof}

Let $(x,a)$ and $(y,b)$ be two vertices of $G_q$ with $x \neq y$, where $d(x,y) = k$ in $G$.
Let $x,x_1,x_2,\dots,x_k = y$ be a geodesic path of length $k$ in $G$.
Then  $(x,a),(x_1,b),(x_2,b),\dots,(x_k,b) = (y,b)$ is a path in $G_q$, implying that \break \hfill
$d((x,a),(y,b)) \leq k = d(x,y)$.

Now suppose that $d((x,a),(y,b)) = l$ in $G_q$, and let \break \hfill $(x,a),(x_1,a_1),(x_2,a_2),\dots,(x_l,a_l) = (y,b)$ be a geodesic path in $G_q$.  Then $x,x_1,x_2,\dots,x_l$ is a walk of length $l$ from $x$ to $y$ in $G$.  Therefore,  $d(x,y) \leq l = d((x,a),(y,b))$, so
$d((x,a),(y,b)) = d(x,y)$.

If $x = y$ and $a \neq b$, than $(x,a)$ and $(x,b)$ are not adjacent in $G_q$, so $d((x,a),(x,b)) > 1$.  Since $G$ is connected with at least two vertices, $x$ has some neighbor $z$ in $G$.  Then both $(x,a)$ and $(x,b)$ are adjacent to $(z,a)$, so $d((x,a),(x,b)) = 2$.

The proof for $G_q^+$ is similar, except $d((x,a),(x,b)) = 1$ since they are adjacent.

\end{proof}

\begin{figure}[h!]\label{fig:qcliqueextension}
\begin{center}
\includegraphics[scale=0.25]{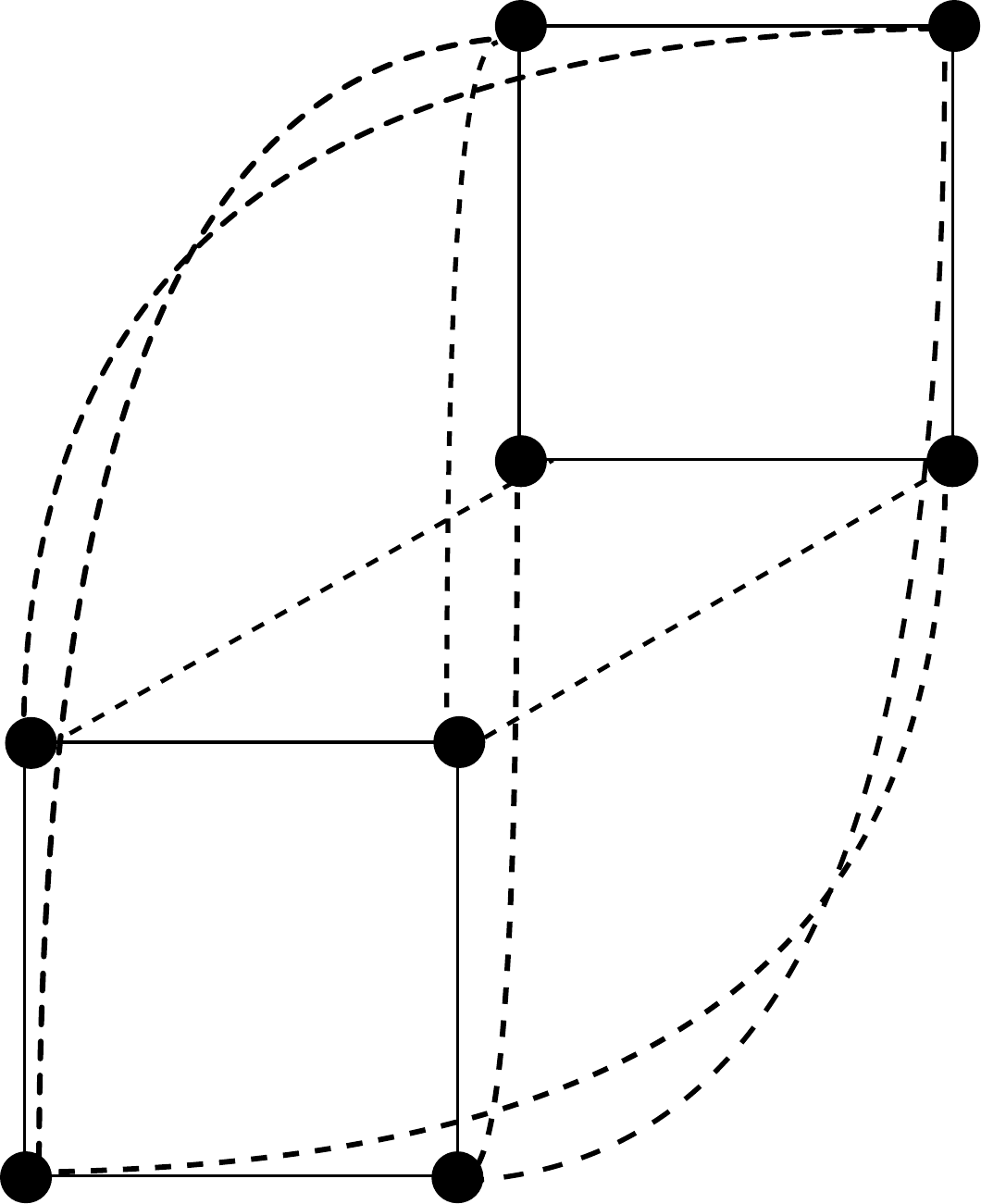}
\end{center}
\caption {The $2$-clique extension of $C_4$.}
\end{figure}

\begin{theorem}\label{cosp}

Let $G,H$ be two $D$-cospectral graphs.  Then $G_q$ and $H_q$ are $D$-cospectral, and
$G_q^+$ and $H_q^+$ are $D$-cospectral.

\end{theorem}

\begin{proof}

Let $D_1$ and $D_2$ be the distance matrices of $G$ and $H$, respectively.
Then the distance matrices of $G_q$ and $H_q$ are $E_1 = J_q \otimes D_1 + (J-I)_q \otimes 2I$ and $E_2 = J_q \otimes D_2 + (J-I)_q \otimes 2I$, respectively.  Noting that $I_q \otimes 2I$ is a scalar multiple of the identity matrix, we have that $E_1$ and $E_2$ are cospectral if and only if $E_1+I_q \otimes 2I$ is cospectral with $E_2 + I_q \otimes 2I$.

We have $E_1+I_q \otimes 2I = J_q \otimes D_1 + (J-I)_q \otimes 2I +I_q \otimes 2I =  J_q \otimes D_1 + J_q \otimes 2I  = J_q \otimes (D_1+2I)$.  Similarly, $E_2+I_q \otimes 2I =
J_q \otimes (D_2+2I)$.  These matrices are cospectral if and only if $D_1+2I$ and $D_2+2I$ are cospectral, which in turn are cospectral if and only if $D_1$ and $D_2$ are cospectral.

The proof for $G_q^+$ and $H_q^+$ is similar.

\end{proof}

We can now prove the following theorem, which answers Questions \ref{Wiener} and \ref{diameter}.

\begin{theorem}

There are infinitely many pairs of $D$-cospectral graphs that have different diameters and Wiener indices.

\end{theorem}

\begin{proof}

Let $G_1$ and $G_2$ be the two $D$-cospectral graphs in Figure \ref{fig:DCospectralDIFERENTdiameterANDWiener}.
The diameters of $G_1$ and $G_2$ are 2 and  3, and the Wiener indices are 71 and 73, respectively.

For each $q$, $(G_1)_q$ and $(G_2)_q$ are $D$-cospectral by Theorem \ref{cosp}. Using Lemma \ref{dmat}, we find their diameters are 2 and 3 respectively and the Wiener indices are $71q^2+10q$ and $73q^2+10q$, respectively.
\end{proof}

We note that the graphs $(G_1)_q^+$ and $(G_2)_q^+$ give another infinite family of examples.

\begin{remark}
Given a square matrix $M$, adding a scalar multiple of the identity matrix to $M$ preserves the eigenvectors and shifts the eigenvalues of $M$.
\end{remark}

The next result shows that $D$-cospectral graphs have the same Wiener index under some sufficient condition.

\begin{proposition}\label{propo:profiles}
Let $G_1, G_2$ be two $k$-transmission-regular graphs. If  $G_1, G_2$ are $D$-cospectral graphs, then $G_1, G_2$ have the same Wiener index.
\end{proposition}
\begin{proof}  Note that $G_1, G_2$ have the same largest distance eigenvalue $\lambda_1$. Then $\lambda_1=k$ corresponds to the all-ones eigenvector, and
it follows that the Wiener index is $\frac{n \lambda_1}{2}$ (since all the row sums of the distance matrices of $G_1, G_2$ are the same by assumption).
\end{proof}

\begin{remark} Note that in order to find counterexamples for Question \ref{Wiener}, one cannot use vertex transitive graphs nor graphs in association schemes since they satisfy the condition of Proposition \ref{propo:profiles}.
\end{remark}


Note that distance-regular graphs are transmission-regular, hence the next corollary follows directly from Proposition \ref{propo:profiles}.

\begin{corollary} If two distance-regular graphs $G_1$ and $G_2$ are $k$-transmission-regular, then
their Wiener indices are the same.
\end{corollary}

\section{Transmission-regular graphs}\label{sec:transmissionregular}
The Wiener index is also known as the distance of a graph or graph transmission. The \emph{transmission index of a vertex} $v$, denoted $T(v)$, is the sum of the entries in the row corresponding to $v$ in the distance matrix of $G$. The \emph{transmission index of a graph} $G$ is the sum of all the transmission indices of its vertices. Given a graph $G$ and vertex $v\in V(G)$, let $S_i(v;G)$ denote the set of vertices of $G$ at distance $i$ from $v$. A transmission-regular graph $G$ has maximum transmission index if the transmission index of its vertices is the greatest possible over all graphs of the same order. Let $diam(G)$ be the diameter of a graph $G$.

\begin{observation}
\label{transmission}
For any vertex $v$ in an $n$-vertex connected graph $G$, we have $n-1\leq T(v)\leq \frac{n(n-1)}{2}$. The lower bound holds with equality only when $v$ is a dominating vertex; the upper bound holds with equality only when $v$ is an end-vertex of the path $P_n$.
\end{observation}

\begin{theorem}\label{prop:transreg}
If $G$ is a connected transmission-regular graph with $n$ vertices, then $n-1\leq T(v)\leq \lfloor\frac{n^2}{4}\rfloor$. For $n>2$, equality for the lower bound holds only when $G$ is the complete graph $K_n$; equality for the upper bound holds only when $G$ is the cycle graph $C_{n}$.
\end{theorem}

\begin{proof}
To prove the lower bound, note that by Observation \ref{transmission}, $T(v)\geq n-1$ for any $v\in V(G)$ and for $n>2$, equality holds only when $G$ is the complete graph $K_n$.

Let us now show the upper bound.
\begin{claim}\label{claim1}
If $G$ is a transmission-regular graph with $n>2$ vertices, $G$ has no cut vertex.
\end{claim}

Suppose for contradiction that $G$ has a cut vertex $v$. Let $G_1$ be the smallest component of $G-v$, $u$ be a vertex in $G_1$ adjacent to $v$, and $G_2=G-G_1-\{v\}$. Note that $d(u,x)=d(v,x)+1$ for $x\in V(G_2)$ since $v$ is a cut vertex, and that $d(u,x)\geq d(v,x)-1$ for $x\in V(G_1)$ since $u$ is adjacent to $v$. Then,

\begin{eqnarray*}
&&T(u)=\sum_{x\in V(G)}d(u,x)=\\
&=&d(u,v)+\sum_{x\in V(G_1)}d(u,x)+\sum_{x\in V(G_2)}d(u,x)\\
&\geq&d(u,v)+\sum_{x\in V(G_1)}(d(v,x)-1)+\sum_{x\in V(G_2)}(d(v,x)+1)\\
&=&1+T(v)-|V(G_1)|+|V(G_2)|>T(v),\\
\end{eqnarray*}

\noindent where the last inequality follows from the assumption that $G_1$ is the smallest component of $G-v$; this contradicts the fact that $G$ is transmission-regular.

Now, let $G$ be a transmission-regular graph with $n>2$ vertices and maximum transmission index. By Claim \ref{claim1}, $G$ cannot have a vertex of degree 1, since then $G$ will have a cut vertex and will not be transmission-regular. If every vertex of $G$ has degree 2, then $G$ is a cycle, which is transmission-regular and has $T(v)=\lfloor\frac{n^2}{4}\rfloor$. Now suppose for contradiction that $G$ has a vertex $v$ of degree greater than 2. Let $d=\max_{u\in V}d(u,v)$, i.e., $d$ is the largest index $i$ for which $S_i(v,G)$ is nonempty. If $|S_i(v,G)|=1$ for any $i\in \{1,\ldots,d-1\}$, then $G$ has a cut vertex --- a contradiction to Claim \ref{claim1}. Thus, $|S_1(v,G)|\geq 3$ and $|S_i(v,G)|\geq 2$ for all $i\in \{2,\ldots,d-1\}$; moreover, $d\leq (n-1)/2$. Let $v'$ be a vertex of a cycle $C_n$. Then,

\begin{eqnarray*}
T(v)=\sum_{i=1}^d|S_i(v,G)|\cdot i&=&
\sum_{i=1}^{d-1}2i+\sum_{i=1}^{d-1}(|S_i(v,G)|-2)i+|S_d(v,G)|d\\
&<&\sum_{i=1}^{d-1}2i+\sum_{i=1}^{d-1}(|S_i(v,G)|-2)d+|S_d(v,G)|d\\
&=&d(d-1)+d\sum_{i=1}^d|S_i(v,G)|-2d(d-1)\\
&=&d(d-1)+d(n-1)-2d(d-1)\\
&=&dn-d^2\leq \left\lfloor \frac{n^2}{4}\right\rfloor =T(v').
\end{eqnarray*}
\noindent This contradicts the assumption that $G$ has maximum transmission index. Thus, $G$ cannot have vertices of degree greater than 2, so the only graph with maximum transmission index is the cycle $C_n$. Thus, for an arbitrary transmission-regular graph, $T(v)\leq\lfloor\frac{n^2}{4}\rfloor$.

\end{proof}

The second part of the proof of Theorem \ref{prop:transreg} guarantees the following:

\begin{corollary}
If $G$ is an $n$-vertex connected graph, then for any vertex $v$, $T(v) \le \frac{diam(G)\left(diam(G)-1 \right)}2 + (n-diam(G))diam(G)$,
and equality holds only when $G$ is a graph obtained by joining one end-vertex of a path $P$ on $diam(G)$ vertices to all the vertices of an arbitrary graph $H$ on $n-diam(G)$ vertices.
\end{corollary}

\section{Wiener index of some tree-like graphs}\label{sec:Wtreelikegraphs}
The  Wiener  index  of  trees  and  tree-like  graphs  has been  widely  studied. In this section we will focus on two families of tree-like graphs: block-clique graphs and linear $k$-trees.

Around 1990 it was shown independently in
several papers that there is a connection between the Wiener index $W$ and the
Laplacian eigenvalues $\mu_1\geq \mu_2\geq \cdots \geq \mu_n=0$ of a tree $T$ on $n$ vertices \cite{Me1989,Me1990,M1991}:
\begin{equation}\label{WT}
W(T)=n\sum_{i=1}^{n-1}\frac{1}{\mu_i}.
\end{equation}

 A natural generalization of trees are \emph{block-clique graphs}; these are exactly the connected graphs in which every block (i.e., every maximal 2-connected subgraph) is a clique. In \cite{BJT2010} block-clique graphs were characterized by using its Wiener index. In this section we give a generalization of (\ref{WT}) for block-clique graphs having the same block size.

Given a graph $G$, we define $f:E(G)\rightarrow \mathbb{N}$ where for $e\in E(G)$, $f(e)$ is the number of pairs of vertices $u,v$ of $G$ such that the shortest $u,v$-path contains the edge $e$. We call $f(e)$ the \textit{contribution of  e} in the Wiener index of $G$. Note that $W(G)=\sum_{e\in E(G)} f(e)$.

\begin{theorem}
Let $G$ be a connected graph of order $n$ where each block is a clique of order $b$. Then

	$$W(G)=\frac{nb}{2}\,\sum_{i=1}^{n-1}\frac{1}{\mu_i}$$

	where $\mu_1\geq \mu_2\geq \cdots \geq \mu_n=0$ are the Laplacian eigenvalues of $G$.
\end{theorem}

\begin{proof}
Let $$p(x)=c_nx^n+c_{n-1}x^{n-1}+\cdots +c_2x^2+c_1x+c_0$$
	be the characteristic polynomial of the Laplacian matrix of $G$, where $c_n=1$. By Vieta's formula, we have
	$$\sum_{1\leq i_1<\cdots <i_k\leq  n}\mu_{i_1}\mu_{i_2}\dots \mu_{i_k}=(-1)^k\frac{c_{n-k}}{c_n}=(-1)^{k}c_{n-k} \quad \text{for }k=1,\ldots,n.$$
	In particular,
	$$(-1)^{n-1}\,c_1=\mu_1\cdots \mu_{n-1}$$
	since $\mu_n=0$; moreover,
	$$(-1)^{n-2}c_2=\sum_{1\leq i_1<\cdots <i_{n-2}\leq  n}\mu_{i_1}\mu_{i_2}\dots \mu_{i_{n-2}}.$$
	None of $\mu_1,\dots ,\mu_{n-1}$ is zero, as $G$ is connected. Now, note that
	\begin{equation}\label{c2c1frac1}
	-\frac{c_2}{c_1}=\sum_{i=1}^{n-1}\frac{1}{\mu_i}.
	\end{equation}
	
	It is well known that the coefficients of $p(x)$ have the following combinatorial interpretation for every graph $G$ (see, for example, \cite[p.38]{cds}):
	$$c_k=(-1)^{n-k}\sum_{F\in \mathcal{F}_k(G)}\gamma(F)$$
	where $\mathcal{F}_k(G)$ is the set of all spanning forests of $G$ having exactly $k$ components, and $\gamma(F)=n_1n_2\dots n_k$ where $n_1,n_2,\dots ,n_k$ are the orders of the components of $F$. So,
	$$c_1=(-1)^{n-1}\,n\,(\text{number of spanning trees of G}).$$
	$T$ is a spanning tree of of $G$ if and only if $T$ induces a spanning tree of $B$ for every block $B$ of $G$. The number of spanning trees of a block of $G$ is $b^{b-2}$, since every block of $G$ is a  $b$-clique. Hence, the number of spanning trees of $G$ is $b^{(b-2)r}$ where $r$ is the number of blocks of $G$, implying that
	\begin{equation}\label{c1}
	c_1=(-1)^{n-1}\,n\,b^{(b-2)r}.
	\end{equation}
	By the combinatorial interpretation of the coefficients,
	\begin{equation}
	c_2=(-1)^{n-2}\sum_{F\in \mathcal{F}_2(G)}\gamma(F).
	\end{equation}
	Now, we shall show that
	\begin{equation}\label{c2wiener}
	b^{(b-2)(r-1)}\left(\sum_{i=0}^{b-2}{b-2 \choose i}(i+1)^{i-1}(b-1-i)^{b-3-i}\right)W(G)=\sum_{F\in \mathcal{F}_2(G)} \gamma(F).
	\end{equation}
	Observe that there exists a unique shortest path between every pair of vertices in $G$. If $F\in \mathcal{F}_2(G)$ then there exist exactly one block $B$ of $G$ such that $F$ induces a spanning forest on $B$ with two components. If $u$ and $v$ are two vertices of $F$ which belong to different components of $F$, then $B$ contains exactly one edge which belongs to the shortest $u,v$-path, say $e$. Let $e=u'v'$ be such that the shortest $u,v$-path is in the form of $u\cdots u'v'\cdots v$. The contribution of the edge $e$ in the Wiener index is counted $b^{(b-2)(r-1)}\left(\sum_{i=0}^{b-2}{b-2 \choose i}(i+1)^{i-1}(b-1-i)^{b-3-i}\right)$ many times on the right side of \eqref{c2wiener}; this is because for every block other than $B$, we can pick any spanning tree of the block (there are $b^{(b-2)(r-1)}$ such choices), and because there are $\sum_{i=0}^{b-2}{b-2 \choose i}(i+1)^{i-1}(b-1-i)^{b-3-i}$ many choices to pick a spanning forest of $B$  with two components such that $u'$ and $v'$ belong to different components. To see the latter, consider the components $T_1$ and $T_2$ of the spanning forest of $B$ induced by $F$ such that $u'\in V(T_1)$ and $v'\in V(T_2)$, and $T_1$ and $T_2$ have orders $i+1$ and $b-1-i$ respectively where $0\leq i\leq b-2$. There are ${b-2\choose i}$ choices to partition the vertices of $B$ into two sets each of them containing exactly one of $u'$ and $v'$, and once the partition is determined, there are $(i+1)^{i-1}$ and $(b-1-i)^{b-3-i}$ choices to make the trees $T_1$ and $T_2$ respectively, as $B$ is a clique.
	
	By \eqref{c1} and \eqref{c2wiener}, we get
	$$-\frac{c_2}{c_1}=\frac{1}{n\,b^{b-2}}\left(\sum_{i=0}^{b-2}{b-2 \choose i}(i+1)^{i-1}(b-1-i)^{b-3-i}\right)\,w(G),$$

which can be simplified (cf. \cite{oeis}) to

$$-\frac{c_2}{c_1}=\frac{1}{n\,b^{b-2}}\left(2b^{b-3}\right)\,w(G).$$

	Equating the right sides of the latter and \eqref{c2c1frac1}, we obtain the desired result.
\end{proof}

Note that as a corollary we obtain the result of Mohar and Merris, since a tree is a block-clique graph where every block of has order $2$.
\begin{corollary} \cite{Me1989,Me1990,M1991}
	If $G$ is a tree of order $n$ then
	$$W(G)=n\,\sum_{i=1}^{n-1}\frac{1}{\mu_i}$$
		where $\mu_1\geq \mu_2\geq \cdots \geq \mu_n=0 $ are the Laplacian eigenvalues of $G$.
\end{corollary}

In the vein of the aforementioned results on the Wiener indices of trees and block-clique graphs, it is a general problem of interest to consider other families of graphs with exploitable structure and derive specialized formulas and algorithms for efficient computation of their Wiener indices. Such methods make use of key properties of the considered graphs in order to speed up or avoid calculating the distance between every pair of vertices. For instance, Gray and Wang \cite{gray_wang} give formulas and bounds for the Wiener index of unicyclic graphs and related families; Chen et al. \cite{CDF2012} give an algorithm for computing a variation of the Wiener index in cactus graphs. In the remainder of this section, we give an algorithm for computing the Wiener index of linear $k$-trees.

A \emph{$k$-tree} is a graph that can be constructed recursively by starting with a copy of $K_{k+1}$ and connecting each new vertex to the vertices of an existing $k$-clique. A \emph{linear $k$-tree} is either a graph isomorphic to $K_{k+1}$, or a $k$-tree with exactly two vertices of degree $k$. A linear $k$-tree can be constructed recursively by starting with a copy of $K_{k+1}$ and connecting each new vertex to the vertices of an existing $k$-clique which has a vertex of degree $k$. A \emph{recursive labeling} of a linear $k$-tree is a labeling of its vertices so that the vertices with labels $1,\ldots,k+1$ form a $k+1$-clique, and every vertex with label $j>k+1$ is adjacent to exactly $k$ vertices with labels smaller than $j$; for $j>k+1$, let $\tilde{N}(j)$ be the set of vertices with labels smaller than $j$ that vertex $j$ is connected to. The recursive labeling of a linear $k$-tree reflects the order in which it can be constructed according to its recursive definition.

Linear $k$-trees and their subgraphs have been used to characterize forbidden minors for certain values of the Colin de Verdi\`{e}re-type analog of the maximum nullity of a graph $G$ (see \cite{H2007,HH2007,JLS2009}) and are also related to the treewidth, zero-forcing number, and proper path-width of $G$ (see \cite{Betal2013,TUK1994}). See also \cite{Markenzon} for additional characterizations and structural properties of linear $k$-trees and related classes of graphs. Note that the shortest path between two vertices in a linear $k$-tree is not necessarily unique.

It is known that $k$-trees (and hence linear $k$-trees) are the maximal graphs with treewidth $k$. In \cite{CK2009}, a near-linear time algorithm is given for computing the Wiener index of graphs with fixed treewidth. However, the runtime of this algorithm is exponential in the treewidth, making it infeasible for graphs with large but fixed treewidth.

\begin{proposition}
The Wiener index of a linear $k$-tree can be computed in $O(kn^2)$ time with Algorithm 1.
\begin{algorithm2e}
\textbf{if} $G\simeq K_n$ \textbf{then return} $W(G)=n(n-1)/2$\;
Initialize $\tilde{N}(v_{n+1})=V(G)$ and $D(G)$ as all-zero $n\times n$ matrix\;
\For{$i = n$ \emph{\KwTo} $k+2$}{
Find a vertex in $\tilde{N}(v_{i+1})$ of degree $k$, label it $v_i$, store its neighbors as $\tilde{N}(v_i)$, and delete $v_i$\;
}
Label remaining vertices $v_1,\ldots,v_{k+1}$; \quad[$v_1,\ldots,v_n$ is recursive labeling]

\For{$\ell = 1$ \emph{\KwTo} $n$}{
\For{$i = \ell + 1$  \emph{\KwTo} $k + 1$}{
  $D[v_i,v_\ell] = D[v_\ell,v_i] = 1$\;
}

\For{$i=\max\{\ell, k+1\}+1$ \emph{\KwTo} $n$}{
   $D[v_i,v_\ell] = D[v_\ell,v_i] = \min_{v_j \in \tilde{N}(i)} \{D[v_\ell,v_j]\} + 1$\;
}
}
\Return $W(G)=\frac{1}{2}\sum_{i=1}^n\sum_{j=1}^n D[i,j]$\;

\NoCaptionOfAlgo
\caption{Algorithm 1: Computing $W(G)$ for a linear $k$-tree $G$}
\end{algorithm2e}
\end{proposition}
\begin{proof}
Let $G$ be a linear $k$-tree.
If $G\simeq K_n$, it is known that $W(G)=(n-1)+\ldots+1=n(n-1)/2$. Thus, assume henceforth that $G\not\simeq K_n$ and fix a recursive labeling $v_1,\ldots,v_n$ on the vertices of $G$. Note that since $|\tilde{N}(v_i)|=k$, the labeling and the list of neighborhoods $\{\tilde{N}(v_i)\}_{i=k+2}^n$ can be created in $O(kn)$ time and maintained.


From the recursive construction of the graph, it follows that the shortest path between a vertex $v$ and a vertex $u$ with a smaller label than $v$ does not pass through a vertex $w$ with a larger label than $v$, since any neighbor with a smaller label than $v$ of a vertex with a larger label than $v$ is also a neighbor of $v$. Thus, for $2\leq i\leq k+1$, $d(v_1,v_i)=1$ and for $k+2\leq i\leq n$,

\begin{equation*}
d(v_1,v_i)=\min_{v_j\in \tilde{N}(v_i)}\{d(v_1,v_j)\}+1.
\end{equation*}
By maintaining $d(v_1,v_j)$ for all $k+1<j<i$, $d(v_1,v_i)$ can be computed in $O(k)$ time. Thus, the first row of the distance matrix of $G$ can be computed in $O(kn)$ time. Similarly, this procedure can be repeated to find each row of the upper triangle of the distance matrix, by iteratively computing the distance between the current vertex in the recursive labeling and the vertices with larger labels in the graph. Hence, the Wiener index of $G$ can be found in $O(kn^2)$ time.
\end{proof}

We now also provide tight bounds for the Wiener indices of linear $k$-trees and classify the extremal graphs in this family (with respect to Wiener index) for each $k$ and $n$. Note that the given bounds generalize the known closed formulas for the Wiener indices of paths and complete graphs.

\begin{corollary}
Let $G$ be a linear $k$-tree on $n$ vertices. Then,
\begin{equation*}
\frac{k^2+k}{2}+n^2-kn-n\leq W(G)\leq \frac{(j+1)(2j^2k^2+jk(3+k-6n)+6n^2-6n)}{12},
\end{equation*}
where $j=\lfloor\frac{n-1}{k}\rfloor$.
The lower bound holds with equality for any linear $k$-tree with at least one dominating vertex such as the one in Figure \ref{fig:BorisExtremal}, left, and the upper bound holds for the family of linear $k$-trees described in Figure \ref{fig:BorisExtremal}, right.
\end{corollary}

\begin{proof}
Let $G_1$ be a linear $k$-tree with a dominating vertex and a recursive vertex labeling. By construction, the $i^\text{th}$ row of the lower triangle of the distance matrix $D_1$ of $G_1$ has exactly $\min\{k,i-1\}$ entries equal to 1, and the rest of its entries are equal to 2. Thus, the sum of the entries in the lower triangle of $D_1$, i.e. $W(G_1)$, is
\begin{eqnarray*}
&&\sum_{i=1}^n[\min\{k,i-1\}+2(i-1-\min\{k,i-1\})]\\
&=&2\sum_{i=1}^n(i-1)-\sum_{i=1}^n\min\{k,i-1\}\\
&=&n^2-n-\frac{k^2-k}{2}-(n-k-1)k.
\end{eqnarray*}
\noindent Now, let $G$ be any linear $k$-tree with a recursive labeling. The $i^\text{th}$ row of the lower triangle of the distance matrix $D$ of $G$ has exactly $\min\{k,i-1\}$ entries equal to 1, and the rest of its entries are greater than or equal to 2. Thus, $W(G)\geq W(G_1)=\frac{k^2+k}{2}+n^2-kn-n$.

Let $G_2$ be the linear $k$-tree with vertex $i$ adjacent to $i-1,\ldots,i-k$ in the recursive labeling. By construction, the $i^\text{th}$ row of the lower triangle of the distance matrix $D_2$ of $G_2$ has $k$ entries equal to $\ell$ for $1\leq \ell\leq \lfloor\frac{i-1}{k}\rfloor$ and the remaining $i-1\mod k$ entries equal to $\lfloor\frac{i-1}{k}\rfloor+1$. To evaluate the sum of all the entries in the lower triangle of $D_2$, we can group terms diagonally, noting that there are $n-(\ell-1)k-1+\ldots+n-(\ell-1) k-k$ terms equal to $\ell$ for $1\leq \ell\leq j=\lfloor\frac{n-1}{k}\rfloor$ and $\frac{(n-1-jk)(n-1-jk+1)}{2}$ terms equal to $j+1$. Thus, we have
\begin{eqnarray*}
&&\sum_{\ell=1}^j \ell\left(\sum_{i=1}^k  (n-(\ell-1)k-i)\right)+(j+1)\frac{(n-1-jk)(n-jk)}{2} \\
&=&\sum_{\ell=1}^j \ell\left(nk-k^2(\ell-1)-\frac{k(k+1)}{2}\right)+(j+1)\frac{(n-1-jk)(n-jk)}{2}\\
&=&nk\frac{j(j+1)}{2}-k^2\frac{j(j+1)(2j+1)}{6}+k^2\frac{j(j+1)}{2}-\frac{k(k+1)}{2}\frac{j(j+1)}{2}\\
&&+(j+1)\frac{(n-1-jk)(n-jk)}{2}.
\end{eqnarray*}
\noindent Now, let $G$ be any linear $k$-tree with a recursive labeling. The $i^\text{th}$ row of the lower triangle of the distance matrix $D$ of $G$ has exactly $\min\{k,i-1\}$ entries equal to 1 and at least $\min\{2k,i-1\}$ entries less than or equal to 2, since each vertex in $\tilde{N}(i)$ has degree at least $k+1$. Similarly, it has at least $\min\{k \ell,i-1\}$ entries which are less than or equal to $\ell$ for $1\leq \ell\leq \lfloor\frac{i-1}{k}\rfloor$. By induction, it follows that the entries in the $i^\text{th}$ row of the lower triangle of $D$ are maximized when there are $k$ entries equal to $\ell$ for $1\leq \ell\leq \lfloor\frac{i-1}{k}\rfloor$ and the remaining $i-1\mod k$ entries equal to $\lfloor\frac{i-1}{k}\rfloor+1$, which is the combination realized by $D_2$. Thus, $W(G)\leq W(G_2)=\frac{(j+1)(2j^2k^2+jk(3+k-6n)+6n^2-6n)}{12}$.
\end{proof}

\begin{figure}[h!]\label{fig:BorisExtremal}
\begin{center}
\includegraphics[scale=0.35]{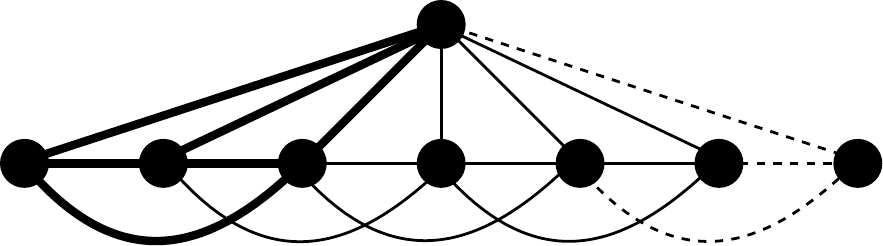}\qquad\qquad\includegraphics[scale=0.35]{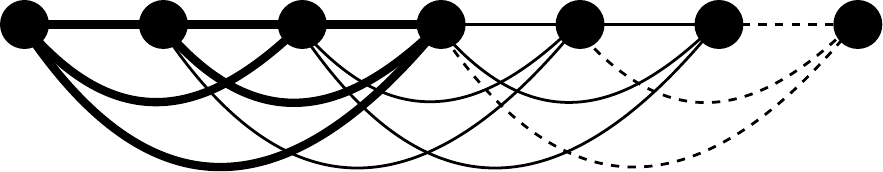}\qquad\qquad
\caption{Linear $k$-trees with extremal Wiener indices, obtained by starting from a copy of $K_{k+1}$ shown in bold, and appending vertices as shown by the dotted lines.}
\end{center}
\end{figure}

\subsection*{Acknowledgments}

We thank Kristin Heysse for her help with some SAGE simulations, and gratefully acknowledge
financial support for this research from the following grants and organizations: NSF-DMS Grants 1604458, 1604773, 1604697 and 1603823 (B. Brimkov, A. Erey, L. Leshock, X. Mart\'{i}nez-Rivera, S.-Y. Song, J. Williford), The Combinatorics Foundation (A. Abiad, S. O), Institute of Mathematics and its Applications (X. Mart\'{i}nez-Rivera), NSF 1450681 (B. Brimkov), NSF 1400281 (J. Williford), KRF-MSIP grant 2016R1A5A1008055 (S. O). We would also like to thank the anonymous referees for their valuable comments and suggestions, which helped to improve this paper.

%
%
%
%


\begin{thebibliography}{99}

\bibitem{AH14} M. Aouchiche and P. Hansen, Distance spectra of graphs: A survey, {\em Linear Algebra Appl.} \textbf{458} (2014),
301--386.

\bibitem{AP15}
F. Atik and P. Panigrahi,
Graphs with few distinct distance eigenvalues irrespective of the diameters,
{\em Electron. J. Linear Algebra}  \textbf{29}  (2015), 194--205.



\bibitem{Balakrishnan2009}
K. Balakrishnan, M. Changat, I. Peterin, S. \v{S}pacapan, P. \v{S}parl, and A.R. Subhamathi,
Strongly distance-balanced graphs and graph products,
\emph{European J. Combin.} \textbf{30(5)} (2009), 1048--1053.


\bibitem{Betal2013} F. Barioli, W. Barrett, S.M. Fallat,
H.T. Hall, L. Hogben, B. Shader, P. van den
Driessche, and H. van der Holst, Parameters Related to Tree-Width, Zero Forcing, and Maximum Nullity of a Graph, \emph{J. Graph Theory} \textbf{72.2} (2013), 146--177.

\bibitem{BJT2010} A. Behtoei, M. Jannesari, and B. Taeri, A characterization of block graphs, \emph{Disc. Appl. Math.} \textbf{158} (2010), 219-221.

\bibitem{CK2009} S. Cabello and C. Knauer, Algorithms for graphs of bounded treewidth via orthogonal range searching, \emph{Computational Geometry} \textbf{42(9)} (2009), 815--824.



\bibitem{CDF2012} N. Chen, W.-X. Du, and Y.-Z. Fan, On Wiener polarity index of cactus graphs, preprint, \href{https://arxiv.org/abs/1211.3513}{https://arxiv.org/abs/1211.3513}.

\bibitem{W-Ref-1} Z. Cinkir,
Contraction formulas for the Kirchhoff and Wiener indices,
\emph{MATCH Commun. Math. Comput. Chem.}
\textbf{75(1)} (2016), 169--198.

\bibitem{cds} D. Cvetkovi{\'c}, M. Doob, and H. Sachs, Spectra of graphs -- theory and application, Barth, Heidelberg, 1995.


\bibitem{D2012} P. Dankelmann, Average distance in weighted graphs, \emph{Disc. Math.} \textbf{312} (2012), 12--20.

\bibitem{Survey-Wiener2001}
A.A. Dobrynin, R. Entringer, I. Gutman,
Wiener index of trees: theory and applications,
\emph{Acta. Appl. Math.} \textbf{66} (2001), 211--249.




\bibitem{W-Ref-2} M. Goubko,
Minimizing Wiener index for vertex-weighted trees with given weight and degree sequences,
\emph{MATCH Commun. Math. Comput. Chem.}
\textbf{75(1)} (2016), 3--27.

\bibitem{GP1971}R.L. Graham and H.O. Pollak,
On the addressing problem for loop switching,
\emph{Bell Syst. Tech. J.} \textbf{50} (1971), 2495--2519.

\bibitem{gray_wang} D. Gray and H. Wang, Cycles, the Degree Distance, and the Wiener Index, \emph{Open Journal of Discrete Mathematics} \textbf{2} (2012), 156--159.




\bibitem{Handa1999} K. Handa,
Bipartite graphs with balanced $(a,b)$-partitions,
\emph{Ars Combin.} \textbf{51} (1999), 113--119.

\bibitem{K2016} K. Heysse, A construction of distance cospectral graphs, preprint, \href{https://arxiv.org/abs/1606.06782}{https://arxiv.org/abs/1606.06782}

\bibitem{HH2007} L. Hogben and H. van der Holst, Forbidden minors for the class of graphs $G$ with $\xi(G)\leq 2$, {\em Linear Algebra Appl.} \textbf{423} (2007), 42--52.




\bibitem{H2007} H. van der Holst, Three-connected graphs whose maximum nullity is at most three, {\em Linear Algebra Appl.} \textbf{429} (2007), 625--632.

\bibitem{Hosoya1971} H. Hosoya,
Topological index. A newly proposed quantity characterizing the topological nature of structural isomers of saturated hydrocarbons,
\emph{Bull. Chem. Soc. Jpn.} \textbf{4} (1971),
2332–-2339.



\bibitem{dist-balanced}
A. Ili\'{c}, S. Klav\v{z}ar, and M. Milanovi\'{c},
On distance-balanced graphs,
\emph{European J. Combin.} \textbf{31} (2010), 733--737.

\bibitem{I2009} G. Indulal,
Sharp bounds on the distance spectral radius and the distance energy of graphs,
\emph{Linear Algebra Appl.} \textbf{430} (2009), 103--116.


\bibitem{I2009coclique} G. Indulal,
Distance spectrum of graph compositions,
\emph{Ars Math. Contemp.} \textbf{2} (2009), 93--100.



\bibitem{JLS2009} C.R. Johnson, R. Loewy, and P.A. Smith, The graphs for which maximum multiplicity of an eigenvalue is
two, \emph{Linear Mult. Alg.} \textbf{57} (2009), 713--736.



\bibitem{W-Ref-3} A. Kelenc, S. Klav\v{z}ar, and N. Tratnik,
The edge-wiener index of benzenoid systems in linear time,
\emph{MATCH Commun. Math. Comput. Chem.}
\textbf{74(3)} (2015), 521--532.

\bibitem{W-Ref-4}
M. Knor, B. Lu\v{z}ar, R. \v{S}krekovski, and I. Gutman,
On Wiener index of common neighborhood graphs,
\emph{MATCH Commun. Math. Comput. Chem.}
\textbf{72(1)} (2014), 321--332.

\bibitem{KST2016} M. Knor, R. \v{S}krekovski, and A. Tepeh, Mathematical aspects of Wiener index, \emph{Ars Math. Contemp.} \textbf{11} (2016), 327--352.


\bibitem{W-Ref-5} M. Knor, R. \v{S}krekovski, and A. Tepeh,
Orientations of graphs with maximum Wiener index,
\emph{Discrete Appl. Math.}
\textbf{211} (2016), 121--129.


\bibitem{W-Ref-6} H. Lin,
On the Wiener index of trees with given number of branching vertices,
\emph{MATCH Commun. Math. Comput. Chem.}
\textbf{72(1)} (2014), 301--310.


\bibitem{W-Ref-7} H. Lin,
Extremal Wiener index of trees with given number of vertices of even degree,
\emph{MATCH Commun. Math. Comput. Chem.}
\textbf{72(1)} (2014), 311--320.


\bibitem{Markenzon}
L. Markenzon, C.M. Justel, and N. Paciornik, Subclasses of k-trees: Characterization and recognition, \emph{Disc. Appl. Math.} \textbf{154 (5)} (2006), 818--825.

\bibitem{MK1977} B.D. McKay, On the spectral characterisation of trees, \emph{Ars Comb.} \textbf{3} (1977), 219--223.

\bibitem{Me1989} R. Merris, An edge version of the matrix-tree theorem and the Wiener index, \emph{Linear Mult. Alg.} \textbf{25} (1989), 291--296.
\bibitem{Me1990} R. Merris, The distance spectrum of a tree, \emph{J. Graph Theory} \textbf{14} (1990), 365–-369.


\bibitem{M1991} B. Mohar, Eigenvalues, diameter, and mean distance in graphs, \emph{Graphs Combin.} \textbf{7} (1991), 53--64.


\bibitem{S16}
Y. Shang,
Estimating the distance Estrada index,
{\em Kuwait J. Sci.}  \textbf{43}  (2016), 14--19.



\bibitem{oeis}
N. Eaton, W. Kook and L. Thoma,
Sequence A089104 in \emph{Online Encyclopedia of Integer Sequences} (2004).


\bibitem{SI2011} D. Stevanovi\'c and A. Ilic, Spectral properties of Distance Matrix of Graphs, in: I. Gutman, B. Furtula, Distance in Molecular Graphs - Theory, MCM Vol. 12, University of Kragujevac, Kragujevac, 2011, pp. 139-176.


\bibitem{W-Ref-8} H.S. Ramane and V.V. Manjalapur,
Note on the bounds on Wiener number of a graph,
\emph{MATCH Commun. Math. Comput. Chem.}
\textbf{76(1)} (2016), 19--22.


\bibitem{W and dist reg} J.A. Rodr\'iguez,
On the Wiener index and the eccentric distance sum of hypergraphs,
\emph{MATCH Commun. Math. Comput. Chem.}
\textbf{54(1)} (2005), 209--220.



\bibitem{W-Ref-9} R. \v{S}krekovski and I. Gutman,
Vertex version of the Wiener theorem,
\emph{MATCH Commun. Math. Comput. Chem.}
\textbf{72(1)} (2014), 295--300.

\bibitem{TUK1994} A. Takahashi, S. Ueno, and Y. Kajitani, Minimal acyclic forbidden minors for the family of graphs with bounded
path-width, \emph{Disc. Math.} \textbf{127} (1994), 293--304.


\bibitem{Wiener1947}
H. Wiener,
Structural Determination of Paraffin Boiling Points, \emph{J. Am. Chem. Soc.} \textbf{69} (1947), 17--20.


\bibitem{Z07}
B. Zhou,
On the largest eigenvalue of the distance matrix of a tree,
{\em MATCH Commun. Math. Comput. Chem.}  \textbf{58}  (2007), 657--662.



\bibitem{ZT07}
B. Zhou and N. Trinajsti\'c,
On the largest eigenvalue of the distance matrix of a connected graph,
{\em Chemical Physics Letters}  \textbf{447}  (2007), 384--387.



\end{thebibliography}
\end{document}